\documentclass[12pt]{amsart}
\usepackage{amssymb}
\usepackage{amsmath}
\usepackage[active]{srcltx}
\usepackage{t1enc}
\usepackage[latin2]{inputenc}
\usepackage{verbatim}
\usepackage{amsmath,amsfonts,amssymb,amsthm}
\usepackage[mathcal]{eucal}
\usepackage{enumerate}
\usepackage[centertags]{amsmath}
\usepackage{graphics}

\newtheorem{theorem}{Theorem}

\newtheorem{lemma}{Lemma}
\newtheorem{remark}{Remark}

\newtheorem{corollary}{Corollary}

\begin{document}
\author{G. Tutberidze}
\title[$T$ means]{Sharp $\left( H_{p},L_{p}\right) $ type
inequalities of maximal operators of $T$ means with respect to Vilenkin systems with monotone coefficients}
\address{G. Tutberidze, The University of Georgia, School of Science and Technology 77a Merab Kostava St, Tbilisi, 0128, Georgia.}
\email{g.tutberidze@ug.edu.ge}
\thanks{The research was supported by Shota Rustaveli National Science Foundation grant no.FR-19-676.}
\date{}
\maketitle

\begin{abstract}
In this paper we prove and discuss some new $\left( H_{p},L_{p}\right) $
type inequalities of maximal operators of  $T$ means with respect to the Vilenkin systems with
monotone coefficients. We also apply these inequalities to prove
strong convergence theorems of such $T$ means. We also show that these
results are the best possible in a special sense. As applications, both
some well-known and new results are pointed out.
\end{abstract}



\section{Introduction}

The definitions and notations used in this introduction can be
found in our next Section. 

It is well-known that Vilenkin systems do not form bases in the space $L_{1}$%
. Moreover, there is a function in the Hardy space $H_p$, such that the
partial sums of $f$ are not bounded in $L_p$-norm, for $0<p\leq 1.$ Approximation properties of Vilenkin-Fourier series with respect to one- and two-dimensional case can be found in \cite{PTW2} and \cite{tep9}.  Simon \cite{Si4} proved
that there exists an absolute constant $c_{p},$ depending only on $p,$ such
that the inequality 
\begin{equation*}
\frac{1}{\log ^{\left[ p\right] }n}\overset{n}{\underset{k=1}{\sum }}\frac{%
	\left\Vert S_{k}f\right\Vert _{p}^{p}}{k^{2-p}}\leq c_{p}\left\Vert
f\right\Vert _{H_{p}}^{p}\ \ \ \left( 0<p\leq 1\right)
\end{equation*}%
holds for all $f\in H_{p}$ and $n\in \mathbb{N}_{+},$ where $\left[ p\right] 
$ denotes the integer part of $p.$ For $p=1$ analogous results with respect
to more general systems were proved in Blahota \cite{b} and Gát \cite{Ga1} and for $
0<p<1 $ simpler proof was given in Tephnadze \cite{tep8}. Some new strong convergence result for partial sums with respect to Vilenkin system was considered in Tutberidze \cite{tut1}. 

In the one-dimensional case the weak (1,1)-type inequality for the maximal operator of Fejér means 
$
\sigma ^{\ast }f:=\sup_{n\in \mathbb{N}}\left\vert \sigma _{n}f\right\vert
$
can be found in Schipp \cite{Sc} for Walsh series and in Pál, Simon \cite{PS}
for bounded Vilenkin series. Fujji \cite{Fu} and Simon \cite{Si2} verified
that $\sigma ^{\ast }$ is bounded from $H_{1}$ to $L_{1}$. Weisz \cite{We2}
generalized this result and proved boundedness of $\ \sigma ^{\ast }$ \ from
the martingale space $H_{p}$ to the space $L_{p},$ for $p>1/2$. Simon \cite%
{Si1} gave a counterexample, which shows that boundedness does not hold for $%
0<p<1/2.$ A counterexample for $p=1/2$ was given by Goginava \cite{Go} (see
also Tephnadze \cite{tep1}). Moreover, Weisz \cite{We4} proved that the maximal
operator of the Fejér means $\sigma ^{\ast }$ is bounded from the Hardy
space $H_{1/2}$ to the space $weak-L_{1/2}$. In \cite{tep2} and \cite{tep3} 
the following result  was proved:

\textbf{Theorem T1:} Let $0<p\leq 1/2.$ Then the following weighted maximal operator of Fejér means 
\begin{equation*}
\widetilde{\sigma}_p^{\ast }f:=\sup_{n\in \mathbb{N}_+}\frac{\left\vert\sigma_n f \right\vert}{\left(n+1\right)^{1/p-2}\log ^{2\left[1/2+p\right]} \left( n+1\right) }
\end{equation*}
is bounded from the martingale Hardy space $H_{p}$ to the Lebesgue space $L_{p}$. 

Moreover, the rate of the weights $\left\{ 1/\left( n+1\right)
^{1/p-2}\log ^{2\left[ p+1/2\right] }\left( n+1\right) \right\}
_{n=1}^{\infty }$ in $n$-th Fejér mean was given exactly. 

Similar results with respect to Walsh-Kachmarz systems were considered in \cite{gn} for $p=1/2$ and in \cite{tep4} for $0<p<1/2.$ Approximation properties of Fejér means  with respect to Vilenkin and Kaczmarz systems can be found in Tephnadze \cite{tep5}, Tutberidze \cite{tut2}, Persson, Tephnadze and Tutberidze \cite{PTT1}.

 In \cite{bt1} it was proved that there exists an absolute constant $c_{p}$,
depending only on $p$, such that the inequality 
\begin{equation}\label{3cc}
\frac{1}{\log ^{\left[ 1/2+p\right] }n}\overset{n}{\underset{k=1}{\sum }}%
\frac{\left\Vert \sigma _{k}f\right\Vert _{p}^{p}}{k^{2-2p}}\leq
c_{p}\left\Vert f\right\Vert _{H_{p}}^{p}\ \ \ \left( 0<p\leq 1/2,\
n=2,3,\dots \right) .  
\end{equation}%
holds. Some new strong convergence result for Vilenkin-Fejér means was considered \cite{PTTW1}.

Móricz and Siddiqi \cite{Mor} investigated the approximation properties of
some special Nörlund means of Walsh-Fourier series of $L_{p}$ function in
norm. In the two-dimensional case approximation properties of Nörlund
means was considered by Nagy \cite{nag,n,nagy}. In \cite{ptw} it
was proved that the maximal operators of Nörlund means 
$
t^{\ast }f:=\sup_{n\in 
	\mathbb{N}
}\left\vert t_{n}f\right\vert
$ either with non-decreasing coefficients, or non-increasing coefficients, satisfying condition 
\begin{equation} \label{cond0}
\frac{1}{Q_{n}}=O\left(\frac{1}{n}\right),\text{ \ as \ }
n\rightarrow \infty 
\end{equation}
 are bounded from the Hardy space $H_{1/2}$
to the space $weak-L_{1/2}$ and are not bounded from
the Hardy space $H_{p}$ to the space $L_{p},$ when $0<p\leq 1/2.$

In \cite{PTW3} it was proved that for  $0<p<1/2,$ $f\in H_{p}$ and non-decreasing sequence  $\{q_{k}:k\geq0\}$  there exists an absolute constant $c_{p},$ depending only on $p,$ such that the inequality holds
	\begin{equation*}
	\overset{\infty }{\underset{k=1}{\sum }}\frac{\left\Vert t_{k}f\right\Vert
		_{p}^{p}}{k^{2-2p}}\leq c_{p}\left\Vert f\right\Vert _{H_{p}}^{p}
	\end{equation*}%
Moreover, if $f\in H_{1/2}$ and $\{q_{k}:k\geq 0\}$ be a sequence of non-decreasing
	numbers, satisfying the condition 
	\begin{equation}
	\frac{q_{n-1}}{Q_{n}}=O\left( \frac{1}{n}\right) ,\text{ \ \ as \ \ }\
	n\rightarrow \infty, \label{fn011}
	\end{equation}%
	then there exists an absolute constant $c,$ such that the inequality holds 
	\begin{equation*}
	\frac{1}{\log n}\overset{n}{\underset{k=1}{\sum }}\frac{\left\Vert
		t_{k}f\right\Vert _{1/2}^{1/2}}{k}\leq c\left\Vert f\right\Vert
	_{H_{1/2}}^{1/2} 
	\end{equation*}
	
In \cite{tut3} was proved that the maximal operators  
$T^{\ast }f:=\sup_{n\in \mathbb{N}}\left\vert T_{n}f\right\vert$ of $T$ means
either with non-increasing coefficients, or non-decreasing sequence satisfying condition \eqref{fn011} are bounded from the Hardy space $H_{1/2}$
to the space $weak-L_{1/2}$. Moreover, there exists a martingale and such $T $ means for which boundedness from
the Hardy space $H_{p}$ to the space $L_{p}$ do not hold when $0<p\leq 1/2.$

One of the most well-known mean of $T$ means is Riesz summability. In \cite{tep6} it was proved that the maximal operator $R^*$ of Riesz means is bounded from the Hardy space $H_{1/2}$ to the space $weak-L_{1/2}$ and is not bounded from $H_{p}$ to the space $L_{p},$ for $0<p\leq 1/2.$ There also was proved that Riesz summability has better properties than Fejér means. In particular, the following weighted maximal operators  $$\frac{\log n \vert R_n f \vert }{\left( n+1\right) ^{1/p-2}\log ^{2\left[ 1/2+p\right] }\left( n+1\right) }$$ are bounded from $H_{p}$ to the space $L_{p},$ for $0<p\leq 1/2$ and the rate of weights are sharp. Moreover, in \cite{lptt} was also proved that if $0<p<1/2$ and $f\in H_p(G_m),$ then there exists an absolute constant $c_{p},$ depending only on $p,$ such that the inequality holds:
\begin{equation} \label{star}
\overset{\infty}{\underset{n=1}{\sum}}
\frac{\log^p n \left\Vert R_nf\right\Vert _{H_p}^p}{n^{2-2p}}\leq
c_p\left\Vert f\right\Vert_{H_p}^p
\end{equation}
If we compare strong convergence results, given by  \eqref{3cc} and \eqref{star},  we obtain that Riesz means has better properties then Fejér means, for $0<p<1/2,$ but in the case $p=1/2$ is was not possible to show even similar result for Riesz means as it is proved for Fejér means given by inequality \eqref{3cc}.
	
In this paper we prove and discuss some new $\left( H_{p},L_{p}\right) $
type inequalities of maximal operators  of $T$ means with respect to the Vilenkin systems with monotone coefficients. Moreover, we apply these inequalities to prove
strong convergence theorems of such $T$ means. In particular, we also study strong convergence theorems of $T$ means with non-increasing sequences in the case $p=1/2,$ but  under the condition \eqref{cond0}. For example, this condition is fulfilled  for Fejér means but does not hold for Riesz means. We also show that these inequalities are the best possible in a special sense. As applications, both
some well-known and new results are pointed out. 

This paper is organized as follows: In order not to disturb our discussions
later on some definitions and notations are presented in Section 2. The main
results and some of its consequences can be found in Section 3. For the
proofs of the main results we need some auxiliary Lemmas, some of them are
new and of independent interest. These results are presented in Section 4.
The detailed proofs are given in Section 5.

\section{Definitions and Notation}

Denote by $
\mathbb{N}
_{+}$ the set of the positive integers, $\mathbb{N}:=\mathbb{N}_{+}\cup \{0\}.$ Let $m:=(m_{0,}$ $m_{1},...)$ be a sequence of the positive
integers not less than 2. Denote by 
\begin{equation*}
Z_{m_{k}}:=\{0,1,...,m_{k}-1\}
\end{equation*}
the additive group of integers modulo $m_{k}$.

Define the group $G_{m}$ as the complete direct product of the groups $%
Z_{m_{i}}$ with the product of the discrete topologies of $Z_{m_{j}}`$s.

The direct product $\mu $ of the measures 
$
\mu _{k}\left( \{j\}\right) :=1/m_{k}\text{ \ \ \ }(j\in Z_{m_{k}})
$
is the Haar measure on $G_{m_{\text{ }}}$with $\mu \left( G_{m}\right) =1.$

In this paper we discuss bounded Vilenkin groups,\textbf{\ }i.e. the case
when $\sup_{n}m_{n}<\infty .$

The elements of $G_{m}$ are represented by sequences 
\begin{equation*}
x:=\left( x_{0},x_{1},...,x_{j},...\right) ,\ \left( x_{j}\in
Z_{m_{j}}\right) .
\end{equation*}

Set $e_{n}:=\left( 0,...,0,1,0,...\right) \in G,$ the $n-$th coordinate of
which is 1 and the rest are zeros $\left( n\in 
\mathbb{N}
\right) .$
It is easy to give a basis for the neighborhoods of $G_{m}:$ 
\begin{equation*}
I_{0}\left( x\right) :=G_{m},\text{ \ }I_{n}(x):=\{y\in G_{m}\mid
y_{0}=x_{0},...,y_{n-1}=x_{n-1}\},
\end{equation*}%
where $x\in G_{m},$ $n\in \mathbb{N}.$

If we define $I_{n}:=I_{n}\left( 0\right) ,$\ for \ $n\in \mathbb{N}$ and $\ 
\overline{I_{n}}:=G_{m}$ $\backslash $ $I_{n},$ then%
\begin{equation}
\overline{I_{N}}=\left( \overset{N-2}{\underset{k=0}{\bigcup }}\overset{N-1}{%
\underset{l=k+1}{\bigcup }}I_{N}^{k,l}\right) \bigcup \left( \underset{k=1}{%
\bigcup\limits^{N-1}}I_{N}^{k,N}\right) ,  \label{1.1}
\end{equation}%
where 
\begin{equation*}
I_{N}^{k,l}:=\left\{ 
\begin{array}{l}
\text{ }I_{N}(0,...,0,x_{k}\neq 0,0,...,0,x_{l}\neq 0,x_{l+1\text{ }%
},...,x_{N-1\text{ }},...),\text{ \ for }k<l<N, \\ 
\text{ }I_{N}(0,...,0,x_{k}\neq 0,0,...,,x_{N-1\text{ }}=0,\text{ }x_{N\text{
}},...),\text{ \ \ \ \ \ \ \ \ \ \ \ \ \ \ \ for }l=N.%
\end{array}%
\text{ }\right.
\end{equation*}

\bigskip If we define the so-called generalized number system based on $m$
in the following way : 
\begin{equation*}
M_{0}:=1,\ M_{k+1}:=m_{k}M_{k}\,\,\,\ \ (k\in 
\mathbb{N}
),
\end{equation*}%
then every $n\in 
\mathbb{N}
$ can be uniquely expressed as $n=\sum_{j=0}^{\infty }n_{j}M_{j},$ where $%
n_{j}\in Z_{m_{j}}$ $(j\in 
\mathbb{N}
_{+})$ and only a finite number of $n_{j}`$s differ from zero.

We introduce on $G_{m}$ an orthonormal system which is called the Vilenkin
system. At first, we define the complex-valued function $r_{k}\left(
x\right) :G_{m}\rightarrow 
\mathbb{C}
,$ the generalized Rademacher functions, by%
\begin{equation*}
r_{k}\left( x\right) :=\exp \left( 2\pi ix_{k}/m_{k}\right) ,\text{ }\left(
i^{2}=-1,x\in G_{m},\text{ }k\in 
\mathbb{N}
\right) .
\end{equation*}

Next, we define the Vilenkin system$\,\,\,\psi :=(\psi _{n}:n\in 
\mathbb{N}
)$ on $G_{m}$ by: 
\begin{equation*}
\psi _{n}(x):=\prod\limits_{k=0}^{\infty }r_{k}^{n_{k}}\left( x\right)
,\,\,\ \ \,\left( n\in 
\mathbb{N}
\right) .
\end{equation*}

Specifically, we call this system the Walsh-Paley system when $m\equiv 2.$

The norms (or quasi-norms) of the spaces $L_{p}(G_{m})$ and $%
weak-L_{p}\left( G_{m}\right) $ $\left( 0<p<\infty \right) $ are
respectively defined by 
\begin{equation*}
\left\Vert f\right\Vert _{p}^{p}:=\int_{G_{m}}\left\vert f\right\vert
^{p}d\mu ,\text{ }\left\Vert f\right\Vert _{weak-L_{p}}^{p}:=\underset{%
\lambda >0}{\sup }\lambda ^{p}\mu \left( f>\lambda \right) <+\infty .
\end{equation*}%

The Vilenkin system is orthonormal and complete in $L_{2}\left( G_{m}\right) 
$ (see \cite{Vi}).

Now, we introduce analogues of the usual definitions in Fourier-analysis. If 
$f\in L_{1}\left( G_{m}\right) $ we can define Fourier coefficients, partial
sums and Dirichlet kernels with respect to the Vilenkin system in the usual manner: 
\begin{equation*}
\widehat{f}\left( n\right) :=\int_{G_{m}}f\overline{\psi }_{n}d\mu,\ \ \ \
\ \ \
S_{n}f:=\sum_{k=0}^{n-1}\widehat{f}\left( k\right) \psi _{k},\text{ \ \ }%
D_{n}:=\sum_{k=0}^{n-1}\psi _{k\text{ }},\text{ \ \ }\left( n\in 
\mathbb{N}_{+}\right).
\end{equation*}%

It is well known that if $n\in\mathbb{N},$ then
\begin{equation} \label{8dn}
D_{M_n}\left(x\right)=\left\{ \begin{array}{ll} M_n, & x\in I_n, \\
0, & x\notin I_n. \end{array} \right.
\end{equation}
Moreover, if  $n=\sum_{i=0}^{\infty}n_iM_i,$ and $1\leq s_n\leq m_n-1,$ then we have the following identity: 
\begin{equation} \label{9dn}
D_n=\psi_n\left(\sum_{j=0}^{\infty}D_{M_j}\sum_{k=m_j-n_j}^{m_j-1}r_j^k\right),
\end{equation}

The $\sigma $-algebra generated by the intervals $\left\{ I_{n}\left(
x\right) :x\in G_{m}\right\} $ will be denoted by 
$\digamma _{n}\left( n\in 
\mathbb{N}\right) .$ Denote by 
$f=\left( f^{\left( n\right) },n\in \mathbb{N}\right) $ 
a martingale with respect to 
$\digamma _{n}\left( n\in \mathbb{N}\right) .$ (for details see e.g. \cite{We1}).
The maximal function of a martingale $f$ \ is defined by 
$
f^{\ast }=\sup_{n\in \mathbb{N}}\left\vert f^{(n)}\right\vert .
$
For $0<p<\infty $ \ the Hardy martingale spaces $H_{p}$ consist of all
martingales $f$ for which 
\begin{equation*}
\left\Vert f\right\Vert _{H_{p}}:=\left\Vert f^{\ast }\right\Vert
_{p}<\infty .
\end{equation*}

A bounded measurable function $a$ is called a p-atom, if there exists an interval $I$, such that
\begin{equation*}
\int_{I}ad\mu =0,\text{ \ \ }\left\Vert a\right\Vert _{\infty }\leq \mu
\left( I\right) ^{-1/p},\text{ \ \ supp}\left( a\right) \subset I.
\end{equation*}

If $f=\left( f^{\left( n\right) },n\in 
\mathbb{N}
\right) $ is a martingale, then the Vilenkin-Fourier coefficients must be
defined in a slightly different manner: 
\begin{equation*}
\widehat{f}\left( i\right) :=\lim_{k\rightarrow \infty
}\int_{G_{m}}f^{\left( k\right) }\overline{\psi }_{i}d\mu .
\end{equation*}

Let $\{q_{k}:k\geq 0\}$ be a sequence of non-negative numbers. The $n$-th  $T$ means for a Fourier series of $f$ \ are respectively defined by
\begin{equation} \label{nor}
T_nf:=\frac{1}{Q_n}\overset{n-1}{\underset{k=0}{\sum }}q_{k}S_kf,
\end{equation}
where $Q_{n}:=\sum_{k=0}^{n-1}q_{k}.$
 
It is obvious that  \  \
$
T_nf\left(x\right)=\underset{G_m}{\int}f\left(t\right)F_n\left(x-t\right) d\mu\left(t\right),
$
where \ \  $	F_n:=\frac{1}{Q_n}\overset{n}{\underset{k=1}{\sum }}q_{k}D_k$ \ \ is called the $T$ kernel.

We always assume that $\{q_k:k\geq 0\}$ is a sequence of non-negative numbers and $q_0>0.$ Then the summability method (\ref{nor}) generated by $\{q_k:k\geq 0\}$ is regular if and only if $	\lim_{n\rightarrow\infty}Q_n=\infty.$

If we invoke Abel transformation we get the following identities, which are very important for the investigations of $T$ summability:
\begin{eqnarray} \label{2b}
Q_n&:=&\overset{n-1}{\underset{j=0}{\sum}}q_j =\overset{n-2}{\underset{j=0}{\sum}}\left(q_{j}-q_{j+1}\right) j+q_{n-1}{(n-1)},
\end{eqnarray}

\begin{equation} 	\label{2c}
F_n=\frac{1}{Q_n}\left(\overset{n-2}{\underset{j=0}{\sum}}\left(q_j-q_{j+1}\right) jK_j+q_{n-1}(n-1)K_{n-1}\right).
\end{equation}
and
\begin{equation} 	\label{2cc2}
T_nf=\frac{1}{Q_n}\left(\overset{n-2}{\underset{j=0}{\sum}}\left(q_j-q_{j+1}\right) j\sigma_jf+q_{n-1}(n-1)\sigma_{n-1}f\right).
\end{equation}

Let $\{q_{k}:k\geq 0\}$ be a sequence of nonnegative numbers. The $n$-th Nö%
rlund mean $t_{n}$ for a Fourier series of $f$ \ is defined by 
\begin{equation}
t_{n}f=\frac{1}{Q_{n}}\overset{n}{\underset{k=1}{\sum }}q_{n-k}S_{k}f,
\label{nor0}
\end{equation}%
where $Q_{n}:=\sum_{k=0}^{n-1}q_{k}.$ 

If $q_{k}\equiv 1,$ we respectively define the Fejér means $\sigma _{n}$ and
Kernels $K_{n}$ as follows: 
\begin{equation*}
\sigma _{n}f:=\frac{1}{n}\sum_{k=1}^{n}S_{k}f\,,\text{ \ \ }K_{n}:=\frac{1}{n%
}\sum_{k=1}^{n}D_{k}.
\end{equation*}

It is well-known that (for details see \cite{AVD})
\begin{equation} \label{fn5}
n\left\vert K_{n}\right\vert \leq c\sum_{l=0}^{\left\vert n\right\vert
}M_{l}\left\vert K_{M_{l}}\right\vert  
\end{equation} 
and
\begin{equation} \label{fn6}
\left\Vert K_n\right\Vert_1\leq c<\infty.  
\end{equation}

The well-known example of N\"orlund summability is the so-called $\left(C,\alpha\right)$-mean (Ces\`aro means) for $0<\alpha<1,$ which are defined by
\begin{equation*}
\sigma_n^{\alpha}f:=\frac{1}{A_n^{\alpha}}\overset{n}{\underset{k=1}{
		\sum}}A_{n-k}^{\alpha-1}S_kf,
\end{equation*}
where 
\begin{equation*}
A_0^{\alpha}:=0,\qquad A_n^{\alpha}:=\frac{\left(\alpha+1\right)...\left(\alpha+n\right)}{n!}.
\end{equation*}

We also consider the "inverse" $\left(C,\alpha\right)$-means, which is an example of a $T$-means:
\begin{equation*}
U_n^{\alpha}f:=\frac{1}{A_n^{\alpha}}\overset{n-1}{\underset{k=0}{\sum}}A_{k}^{\alpha-1}S_kf, \qquad 0<\alpha<1.
\end{equation*}

Let $V_n^{\alpha}$ denote
the $T$ mean, where $	\left\{q_0=0, \  q_k=k^{\alpha-1}:k\in \mathbb{N}_+\right\} ,$
that is 
\begin{equation*}
V_n^{\alpha}f:=\frac{1}{Q_n}\overset{n-1}{\underset{k=1}{\sum }}k^{\alpha-1}S_kf,\qquad 0<\alpha<1.
\end{equation*}

The $n$-th Riesz logarithmic mean $R_{n}$ and the Nörlund logarithmic mean
$L_{n}$ are defined by
\begin{equation*}
R_{n}f:=\frac{1}{l_{n}}\sum_{k=1}^{n-1}\frac{S_{k}f}{k}\text{ \ \ \ and  \ \ \ }
L_{n}f:=\frac{1}{l_{n}}\sum_{k=1}^{n-1}\frac{S_{k}f}{n-k},
\end{equation*}%
\ respectively, where $l_{n}:=\sum_{k=1}^{n-1}1/k.$

Up to now we have considered $T$ means in the case when the sequence $\{q_k:k\in\mathbb{N}\}$ is bounded but now we consider $T$ summabilities with unbounded sequence $\{q_k:k\in\mathbb{N}\}$. 

Let $\alpha\in
\mathbb{R}_+,\ \ \beta\in\mathbb{N}_+$ and
$
\log^{(\beta)}x:=\overset{\beta-\text{times}}{\overbrace{\log ...\log}}x.
$
If we define the sequence $\{q_k:k\in \mathbb{N}\}$ by
$	\left\{q_0=0, \ q_k=\log^{\left(\beta \right)}k^{\alpha
}:k\in\mathbb{N}_+\right\},$ 
then we get the class of $T$ means with non-decreasing coefficients:
\begin{equation*}
B_n^{\alpha,\beta}f:=\frac{1}{Q_n}
\sum_{k=1}^{n-1}\log^{\left(\beta\right)}k^{\alpha}S_kf.
\end{equation*}%

We note that $B_n^{\alpha,\beta}$ are
well-defined for every $n \in \mathbb{N}$
\begin{equation*}
B_{n}^{\alpha,\beta}f=\sum_{k=1}^{n-1}\frac{\log^{\left(\beta\right)}k^{\alpha }}{Q_n}S_kf.
\end{equation*}

It is obvious that $\frac{n}{2}\log^{\left(\beta \right)}\frac{n^{\alpha }}{2^{\alpha }}\leq Q_n\leq n\log^{\left(\beta\right)}n^{\alpha}.$ It follows that
\begin{eqnarray} \label{node00}
\frac{q_{n-1}}{Q_n}\leq\frac{c\log^{\left(\beta\right)}n^{\alpha}}{n\log^{\left(\beta\right) }n^{\alpha}}= O\left(\frac{1}{n}\right)\rightarrow 0,\text{ \ as \ }n\rightarrow \infty.
\end{eqnarray}

We also define the maximal operator  of $T$ and Nörlund means by 
\begin{eqnarray*}
	T^{\ast}f:=\sup_{n\in\mathbb{N}}\left\vert T_nf\right\vert, \ \ \ t^{\ast}f:=\sup_{n\in\mathbb{N}}\left\vert t_nf\right\vert.
\end{eqnarray*}

Some well-known examples of maximal operators of $T$ means are the maximal operator of Fejér $\sigma^*$ and Riesz $R^*$ logarithmic means, which are defined by: 
\begin{eqnarray*}
	\sigma^{\ast}f:=\sup_{n\in\mathbb{N}}\left\vert \sigma_{n}f\right\vert, \ \ \ \ \ R^{\ast}f:=\sup_{n\in\mathbb{N}}\left\vert R_{n}f\right\vert.
\end{eqnarray*}

\section{The Main Results and Applications}

Our first main result reads:

\begin{theorem}
	\label{theorem3fejermax2222}Let $0<p\leq 1/2,$ $f\in H_{p}$ and $\{q_{k}:k\geq
	0\}$ be a sequence of non-increasing numbers. Then the maximal operator 
	\begin{equation} \label{Tn100}
	\widetilde{T}_{p}^{\ast }f:=\sup_{n\in \mathbb{N}_{+}}\frac{\left\vert
		T_{n}f\right\vert }{\left( n+1\right) ^{1/p-2}\log ^{2\left[ 1/2+p\right]
		}\left( n+1\right) }
	\end{equation}%
	is bounded from the Hardy space $H_{p}$ to the space $L_{p}.$
\end{theorem}

\begin{corollary}\label{corolary1}
	Let $0<p\leq 1/2$ and $f\in H_{p}$. Then the maximal operator 
	\begin{equation*}
	\widetilde{R}_{p}^{\ast }f:=\sup_{n\in \mathbb{N}_{+}}\frac{\left\vert
		R_{n}f\right\vert }{\left( n+1\right) ^{1/p-2}\log ^{2\left[ 1/2+p\right]
		}\left( n+1\right) }
	\end{equation*}%
	is bounded from the Hardy space $H_{p}$ to the space $L_{p}.$
\end{corollary}

\begin{corollary}\label{corolary2}
	Let $0<p\leq 1/2$ and $f\in H_{p}$. Then the maximal operator 
	\begin{equation*}
	\widetilde{U}^{\alpha,\ast}_{p}f:=\sup_{n\in \mathbb{N}_{+}}\frac{\left\vert
		U^{\alpha}_{n}f\right\vert }{\left( n+1\right) ^{1/p-2}\log ^{2\left[ 1/2+p\right]
		}\left( n+1\right) }
	\end{equation*}%
	is bounded from the Hardy space $H_{p}$ to the space $L_{p}.$
\end{corollary}

\begin{corollary}\label{corolary3}
	Let $0<p\leq 1/2$ and $f\in H_{p}$. Then the maximal operator 
	\begin{equation*}
	\widetilde{V}^{\alpha,\ast}_{p}f:=\sup_{n\in \mathbb{N}_{+}}\frac{\left\vert
		V^{\alpha}_{n}f\right\vert }{\left( n+1\right) ^{1/p-2}\log ^{2\left[ 1/2+p\right]
		}\left( n+1\right) }
	\end{equation*}%
	is bounded from the Hardy space $H_{p}$ to the space $L_{p}.$
\end{corollary}

Next, we consider maximal operators of $T$ means with non-decreasing sequence:

\begin{theorem}
	\label{theorem3fejermax22221}Let $0<p\leq 1/2,$ $f\in H_{p}$ and $\{q_{k}:k\geq
	0\}$ be a sequence of  non-decreasing
	numbers, satisfying the condition 
	\begin{equation}
	\frac{q_{n-1}}{Q_{n}}=O\left( \frac{1}{n}\right) ,\text{ \ \ as \ \ }\
	n\rightarrow \infty .  \label{fn01}
	\end{equation}
	Then the maximal operator 
	\begin{equation} \label{Tn1000}
	\widetilde{T}_{p}^{\ast }f:=\sup_{n\in \mathbb{N}_{+}}\frac{\left\vert
		T_{n}f\right\vert }{\left( n+1\right) ^{1/p-2}\log ^{2\left[ 1/2+p\right]
		}\left( n+1\right) }
	\end{equation}%
	is bounded from the martingale Hardy space $H_{p}$ to the space $L_{p}.$
\end{theorem}

\begin{corollary}
	Let $0<p\leq 1/2,$ $f\in H_{p}$ and $\{q_{k}:k\geq 0\}$ be a sequence of non-decreasing numbers, such that
	\begin{equation} \label{condT1}
	\sup_{n\in \mathbb{N}}q_{n}<c<\infty .
	\end{equation}%
Then
\begin{equation*}
\frac{q_{n-1}}{Q_{n}}\leq \frac{c}{Q_{n}}\leq \frac{c}{q_{0}n}=\frac{c_{1}}{n%
}=O\left( \frac{1}{n}\right),\text{ as }n\rightarrow 0,
\end{equation*}
and weighted maximal operators of such $T$ means, given by \eqref{Tn1000} are bounded from the Hardy space $H_{p}$ to the space $L_{p}.$
\end{corollary}

\begin{corollary}
Let $0<p\leq 1/2$ and $f\in H_{p}.$ Then the maximal operator 
\begin{equation*} 
\widetilde{T}_{p}^{\ast }f:=\sup_{n\in \mathbb{N}_{+}}\frac{\left\vert B^{\alpha,\beta}_nf\right\vert }{\left( n+1\right) ^{1/p-2}\log ^{2\left[ 1/2+p\right]}\left( n+1\right) }
\end{equation*}
is bounded from the martingale Hardy space $H_{p}$ to the space $L_{p}.$
\end{corollary}

\begin{remark} According to Theorem T1 we obtain that  weights  in \eqref{Tn100} and \eqref{Tn1000} are sharp. 
\end{remark}

\begin{theorem}
\label{theorem2fejerstrong}a) Let $0<p<1/2,$ $f\in H_{p}$ and $\{q_{k}:k\geq
0\}$ be a sequence of non-increasing numbers. Then there exists an absolute
constant $c_{p},$ depending only on $p,$ such that the inequality holds:
\begin{equation*}
\overset{\infty }{\underset{k=1}{\sum }}\frac{\left\Vert T_{k}f\right\Vert
_{p}^{p}}{k^{2-2p}}\leq c_{p}\left\Vert f\right\Vert _{H_{p}}^{p}
\end{equation*}%

b)Let $f\in H_{1/2}$ and $\{q_{k}:k\geq 0\}$ be a sequence of non-increasing
numbers, satisfying the condition 
\begin{equation}\label{fn0111}
\frac{1}{Q_{n}}=O\left(\frac{1}{n}\right),\text{ \ \ as \ \ }\
n\rightarrow \infty .  
\end{equation}

Then there exists an absolute constant $c,$ such that the inequality holds:
\begin{equation}
\frac{1}{\log n}\overset{n}{\underset{k=1}{\sum }}\frac{\left\Vert
T_{k}f\right\Vert _{1/2}^{1/2}}{k}\leq c\left\Vert f\right\Vert
_{H_{1/2}}^{1/2}  \label{7nor7}
\end{equation}%
\end{theorem}

\begin{corollary}
	Let $0<p\leq 1/2$ and $f\in H_{p}.$ Then there exists absolute constant $%
	c_{p},$ depending only on $p,$ such that the following inequality holds: 
	\begin{equation*}
	\frac{1}{\log ^{\left[ 1/2+p\right] }n}\overset{n}{\underset{k=1}{\sum }}%
	\frac{\left\Vert \sigma _{k}f\right\Vert _{p}^{p}}{k^{2-2p}}\leq
	c_{p}\left\Vert f\right\Vert _{H_{p}}^{p}.
	\end{equation*}
\end{corollary}

\begin{corollary}
	Let $0<p\leq 1/2$ and $f\in H_{p}.$ Then there exists an absolute constant $%
	c_{p},$ depending only on $p,$ such that the following inequalities hold: 
	\begin{equation*}
	\overset{\infty}{\underset{k=1}{\sum }}%
	\frac{\left\Vert U_k^{\alpha}f\right\Vert _{p}^{p}}{k^{2-2p}}\leq
	c_{p}\left\Vert f\right\Vert _{H_{p}}^{p},  \ \ \ \ \ \  \ 
	\overset{\infty}{\underset{k=1}{\sum }}%
	\frac{\left\Vert V_k^{\alpha}f\right\Vert _{p}^{p}}{k^{2-2p}}\leq
	c_{p}\left\Vert f\right\Vert _{H_{p}}^{p}, \ \ \ \ \ \ \ 
	\overset{\infty}{\underset{k=1}{\sum }}
	\frac{\left\Vert R_kf\right\Vert _{p}^{p}}{k^{2-2p}}\leq
	c_{p}\left\Vert f\right\Vert _{H_{p}}^{p}.
	\end{equation*}
\end{corollary}

\begin{theorem}
	\label{theorem2fejerstrong1}a) Let $0<p<1/2,$ $f\in H_{p}$ and $\{q_{k}:k\geq
	0\}$ be a sequence of non-decreasing numbers. Then there exists an absolute
	constant $c_{p},$ depending only on $p,$ such that the inequality holds:
	\begin{equation*}
	\overset{\infty }{\underset{k=1}{\sum }}\frac{\left\Vert T_{k}f\right\Vert
		_{p}^{p}}{k^{2-2p}}\leq c_{p}\left\Vert f\right\Vert _{H_{p}}^{p}
	\end{equation*}
	
	b)Let $f\in H_{1/2}$ and $\{q_{k}:k\geq 0\}$ be a sequence of non-increasing
	numbers, satisfying the condition \eqref{fn01}.
	Then there exists an absolute constant $c,$ such that the inequality holds:
	\begin{equation}
	\frac{1}{\log n}\overset{n}{\underset{k=1}{\sum }}\frac{\left\Vert
		T_{k}f\right\Vert _{1/2}^{1/2}}{k}\leq c\left\Vert f\right\Vert
	_{H_{1/2}}^{1/2}  \label{7nor}
	\end{equation}%
\end{theorem}

\begin{corollary}
	Let $0<p\leq 1/2,$ $f\in H_{p}$ and $\{q_{k}:k\geq 0\}$ be a sequence of
	non-decreasing numbers, such that
$\sup_{n\in \mathbb{N}}q_{n}<c<\infty.$
	Then condition (\ref{fn01}) is satisfied and for all such $T$ means there
	exists an absolute constant $c,$ such that the inequality (\ref{7nor}) holds.
\end{corollary}

We have already considered the case when the sequence $\{q_{k}:k\geq 0\}$ is
bounded. Now, we consider some Nörlund means, which are generated by a
unbounded sequence $\{q_{k}:k\geq 0\}.$

\begin{corollary}
	Let $0<p\leq 1/2$ and $f\in H_{p}.$ Then there exists an absolute constant $%
	c_{p},$ depending only on $p,$ such that the following inequality holds: 
	\begin{equation*}
	\frac{1}{\log ^{\left[ 1/2+p\right] }n}\overset{n}{\underset{k=1}{\sum }}%
	\frac{\left\Vert B^{\alpha,\beta}_kf\right\Vert _{p}^{p}}{k^{2-2p}}\leq
	c_{p}\left\Vert f\right\Vert _{H_{p}}^{p}.
	\end{equation*}
\end{corollary}

\section{Auxiliary lemmas}

We need the following auxiliary Lemmas:

\begin{lemma}[see e.g. \protect\cite{We3}]
\label{lemma2.1} A martingale $f=\left( f^{\left( n\right) },n\in \mathbb{N}\right) $ is in $H_{p}\left( 0<p\leq 1\right) $ if and only if there exists
a sequence $\left( a_{k},k\in 
\mathbb{N}
\right) $ of p-atoms and a sequence $\left( \mu _{k},k\in \mathbb{N}
\right) $ of real numbers such that, for every $n\in \mathbb{N},$ 
\begin{equation} \label{1}
\qquad \sum_{k=0}^{\infty }\mu _{k}S_{M_{n}}a_{k}=f^{\left( n\right) },\text{
\ \ a.e.,}   \ \ \ \ \text{where} \ \ \ \ \sum_{k=0}^{\infty }\left\vert \mu _{k}\right\vert ^{p}<\infty . 
\end{equation}
Moreover,
\begin{equation*}
\left\Vert f\right\Vert _{H_{p}}\backsim \inf \left( \sum_{k=0}^{\infty
}\left\vert \mu _{k}\right\vert ^{p}\right) ^{1/p}
\end{equation*}
\textit{where the infimum is taken over all decompositions of} $f$ \textit{%
of the form} (\ref{1}).
\end{lemma}

\begin{lemma}[see e.g. \protect\cite{We3}]
\label{lemma2.2} Suppose that an operator $T$ is $\sigma $-sublinear and for
some $0<p\leq 1$%
\begin{equation*}
\int\limits_{\overset{-}{I}}\left\vert Ta\right\vert ^{p}d\mu \leq
c_{p}<\infty ,
\end{equation*}%
for every $p$-atom $a$, where $I$ denotes the support of the atom. If $T$ is
bounded from $L_{\infty \text{ }}$ to $L_{\infty },$ then 
\begin{equation*}
\left\Vert Tf\right\Vert _{p}\leq c_{p}\left\Vert f\right\Vert _{H_{p}},%
\text{ }0<p\leq 1.
\end{equation*}
\end{lemma}

\begin{lemma}[see \protect\cite{gat}] \label{gatkn}
\label{lemma2} Let $n>t,$ $t,n\in \mathbb{N}.$ Then%
\begin{equation*}
K_{M_{n}}\left( x\right) =\left\{ 
\begin{array}{c}
\frac{M_{t}}{1-r_{t}\left( x\right) },\text{ \ }x\in I_{t}\backslash I_{t+1},
\text{ }x-x_{t}e_{t}\in I_{n}, \\ 
\frac{M_{n}-1}{2},\text{ \ \ \ \ \ \ \ \ \ \ \ \ \ \ \ \ \ \ \ \ \ \ \ \ \ \
\ \ \ }x\in I_{n}, \\ 
0,\text{ \ \ \ \ \ \ \ \ \ \ \ \ \ \ \ \ \ \ \ \ \ \ \ \ \ \ \ \ \ \
otherwise.}
\end{array}
\right.
\end{equation*}
\end{lemma}

For the proof of our main results we also need the following new Lemmas of
independent interest:

\begin{lemma}\label{T2} 
	Let $n\in\mathbb{N}$ and $\{q_k:k\in\mathbb{N}\}$ be a sequence either of non-increasing numbers, or non-decreasing numbers satisfying condition \eqref{fn01}. Then
	\begin{equation} 	\label{2e}
	\Vert F_n\Vert_1<c<\infty.
	\end{equation}	
\end{lemma}
\textbf{Proof:}
Let $n\in\mathbb{N}$ and $\{q_k:k\in\mathbb{N}\}$ be a sequence of non-increasing numbers. By combining   \eqref{2b} and \eqref{2cc2} with \eqref{fn6} we can conclude that 
\begin{eqnarray*} 	
	\Vert T_n\Vert_1&\leq& \frac{1}{Q_n} \left(\overset{n-2}{\underset{j=0}{\sum}}\left\vert q_j-q_{j+1}\right\vert j\Vert \sigma_j\Vert_1+q_{n-1}(n-1)\Vert \sigma_{n-1}\Vert_1\right)\\
	&\leq &\frac{c}{Q_n} \left(\overset{n-2}{\underset{j=0}{\sum}}\left(q_j-q_{j+1}\right) j+q_{n-1}(n-1)\right)\leq c<\infty.
\end{eqnarray*}	

Let $n\in\mathbb{N}$ and $\{q_k:k\in\mathbb{N}\}$ be a sequence non-decreasing sequence satisfying condition \eqref{fn01}. Then By using again   \eqref{2b} and \eqref{2cc2} with \eqref{fn6} we find that

\begin{eqnarray*} 	
	\Vert T_n\Vert_1&\leq& \frac{1}{Q_n} \left(\overset{n-2}{\underset{j=0}{\sum}}\left\vert q_j-q_{j+1}\right\vert j\Vert \sigma_j\Vert_1+q_{n-1}(n-1)\Vert \sigma_{n-1}\Vert_1\right)\\
	&\leq &\frac{c}{Q_n} \left(\overset{n-2}{\underset{j=0}{\sum}}\left(q_{j+1}-q_j\right) j+q_{n-1}(n-1)\right)\\
	&=&\frac{c}{Q_n} \left(2q_{n-1}(n-1)-\left(\overset{n-2}{\underset{j=0}{\sum}}\left(q_{j}-q_{j+1}\right) j+q_{n-1}(n-1)\right)\right)
	\\
	&=&\frac{c}{Q_n}(2q_{n-1}(n-1)-Q_n) \leq c<\infty.
\end{eqnarray*}

The proof is complete.

\begin{lemma}\label{lemma0nnT0}
	Let $\{q_k:k\in\mathbb{N}\}$ be a sequence of non-increasing numbers and $n>M_N$. Then
	\begin{equation*}
	\left\vert\frac{1}{Q_n}\overset{n-1}{\underset{j=M_{N}}{\sum }}q_{j}D_j\left( x\right)\right\vert
	\leq\frac{c}{M_N}\left\{\sum_{j=0}^{\left\vert n\right\vert }M_j\left\vert K_{M_j}\right\vert\right\},
	\end{equation*}
\end{lemma}

\begin{proof}
	Since sequence is non-increasing number  we get that
	\begin{eqnarray*}
		&&\frac{1}{Q_n}\left(q_{M_N}+\overset{n-2}{\underset{j=M_N}{\sum }}\left\vert q_j-q_{j+1}\right\vert+q_{n-1}\right)\\ &\leq&\frac{1}{Q_n}\left(q_{M_N}+\overset{n-2}{\underset{j=M_N}{\sum}}\left(q_{j}-q_{j+1} \right)+q_{n-1}\right)\\
		&\leq &
		\frac{2q_{M_N}}{Q_n}\leq\frac{2q_{M_N}}{Q_{M_N+1}}\leq \frac{c}{M_N}.
	\end{eqnarray*}
	
	If we apply Abel transformation and \eqref{fn5} we immediately get that
	\begin{eqnarray*}
		&&\left\vert\frac{1}{Q_n}\overset{n-1}{\underset{j=M_{N}}{\sum }}q_{j}D_j\left( x\right)\right\vert \\
		&=&\frac{1}{Q_n}\left( q_{M_n}K_{{M_n-1}}+\overset{n-2} {\underset{j=M_N}{\sum}}\left( q_j-q_{j+1}\right)K_j+q_{n-1}K_{n-1}\right)\\
		&\leq &  \frac{1}{Q_n}\left( q_{M_n}+\overset{n-2} {\underset{j=M_N}{\sum}}\left\vert q_j-q_{j+1}\right \vert+q_{n-1}\right)\sum_{i=0}^{\left\vert n\right\vert }M_i\left\vert K_{M_i}\right\vert \\
		&\leq &\frac{c}{M_N}\sum_{i=0}^{\left\vert n\right\vert }M_i\left\vert K_{M_i}\right\vert.
	\end{eqnarray*}
	The proof is complete.
\end{proof}

\begin{lemma} \label{lemma5aa}
	Let $x\in I_N^{k,l}, \ \ k=0,\dots,N-1, \ \ l=k+1,\dots ,N$
	and $\{q_k:k\in \mathbb{N}\}$ be a sequence of non-increasing numbers. Then there exists an absolute constant, such that
	\begin{equation*}
	\int_{I_N}	\left\vert\frac{1}{Q_n}\overset{n-1}{\underset{j=M_{N}}{\sum }}q_{j}D_j\left( x-t\right)\right\vert d\mu\left(t\right) \leq\frac{cM_lM_k}{M^2_N}.
	\end{equation*}
\end{lemma}

{\bf Proof}:
Let $x\in I_N^{k,l},$ for $0\leq k<l\leq N-1$ and $t\in I_N.$ First, we observe that $x-t\in $ $I_N^{k,l}.$ Next, we apply Lemmas \ref{gatkn} and  \ref{lemma0nnT0} to obtain that
\begin{eqnarray}  \label{88811}
&&\int_{I_N}	\left\vert\frac{1}{Q_n}\overset{n-1}{\underset{j=M_{N}}{\sum }}q_{j}D_j\left( x-t\right)\right\vert d\mu\left(t\right)  \\ \notag
&\leq & \frac{c}{M_N}\underset{i=0}{\overset{\left\vert n\right\vert }{\sum}}M_i\int_{I_N}\left\vert K_{M_i}\left(x-t\right)\right\vert d\mu
\left(t\right) \\ \notag
&\leq& \frac{c}{M_N}\int_{I_N}\overset{l}{\underset{i=0}{\sum }}M_iM_kd\mu\left(t\right)  \leq \frac{cM_kM_l}{M^2_N}
\end{eqnarray}
and the first estimate is proved.

Now, let $x\in I_N^{k,N}$. Since $x-t\in I_{N}^{k,N}$ for $t\in I_N,$ by
combining \eqref{8dn} and  \eqref{9dn} we have that
$$\left\vert D_{i}\left(x-t\right)\right\vert\leq M_k$$
and
\begin{eqnarray} \label{1111001}
&&\int_{I_N}	\left\vert\frac{1}{Q_n}\overset{n-1}{\underset{j=M_{N}}{\sum }}q_{j}D_j\left( x-t\right)\right\vert d\mu\left(t\right)   \\ \notag
&\leq&\frac{c}{Q_n}\underset{i=0}{\overset{\left\vert n\right\vert } {\sum}}q_i\int_{I_N}\left\vert D_{i}\left(x-t\right)\right\vert d\mu\left(t\right) \\ \notag
&\leq&\frac{c}{Q_n}\overset{\left\vert n\right\vert-1} {\underset{i=0}{\sum }}q_i\int_{I_N}M_kd\mu\left(t\right) \leq\frac{cM_k}{M_N}.
\end{eqnarray}

According to (\ref{88811}) and (\ref{1111001}) the proof is complete.

\begin{lemma}\label{lemma0nnT}
	Let $n>M_N$ and $\{q_k:k\in\mathbb{N}\}$ be a sequence of non-increasing numbers, satisfying condition \eqref{fn0111}. Then
	\begin{equation*}
	\left\vert\frac{1}{Q_n}\overset{n-1}{\underset{j=M_{N}}{\sum }}q_{j}D_j\right\vert
	\leq\frac{c}{n}\left\{\sum_{j=0}^{\left\vert n\right\vert }M_j\left\vert K_{M_j}\right\vert\right\},
	\end{equation*}
	where $c$ is an absolute constant.
\end{lemma}

\begin{proof}
	Since sequence is non-increasing number satisfying condition \eqref{fn0111}, we get that
	\begin{eqnarray*}
		&&\frac{1}{Q_n}\left(q_{M_n}+\overset{n-2}{\underset{j=M_N}{\sum }}\left\vert q_j-q_{j+1}\right\vert+q_{n-1}\right)\\
		 &\leq&\frac{1}{Q_n}\left(q_{M_n}+\overset{n-2}{\underset{j=M_N}{\sum}}\left(q_{j}-q_{j+1} \right)+q_{n-1}\right)\\
		&\leq &
		\frac{2q_{M_N}}{Q_n}\leq \frac{c}{Q_n}\leq \frac{c}{n}.
	\end{eqnarray*}
	
	If we apply \eqref{fn5} we immediately get that
	\begin{eqnarray*}
		&&\left\vert\frac{1}{Q_n}\overset{n-1}{\underset{j=M_{N}}{\sum }}q_{j}D_j\right\vert \\
		&\leq & \left( \frac{1}{Q_n}\left( q_{M_n}+\overset{n-2} {\underset{j=M_N+1}{\sum}}\left\vert q_j-q_{j+1}\right \vert+q_{n-1}\right)\right)\sum_{i=0}^{\left\vert n\right\vert }M_i\left\vert K_{M_i}\right\vert \\
		&\leq &\frac{c}{n}\sum_{i=0}^{\left\vert n\right\vert }M_i\left\vert K_{M_i}\right\vert.
	\end{eqnarray*}
	The proof is complete.
\end{proof}

\begin{lemma} \label{lemma5a}
	Let $x\in I_N^{k,l}, \ \ k=0,\dots,N-2, \ \ l=k+1,\dots ,N-1$
	and $\{q_k:k\in \mathbb{N}\}$ be a sequence of non-increasing numbers,
	satisfying condition \eqref{fn0111}. Then
	\begin{equation*}
	\int_{I_N}\left\vert\frac{1}{Q_n}\overset{n-1}{\underset{j=M_{N}}{\sum }}q_{j}D_j\left( x-t\right)\right\vert d\mu\left(t\right) \leq\frac{cM_lM_k}{nM_N}.
	\end{equation*}
	Let $x\in I_N^{k,N},$ \ \ $k=0,\dots,N-1.$ Then
	\begin{equation*}
	\int_{I_N}\left\vert\frac{1}{Q_n}\overset{n-1}{\underset{j=M_{N}}{\sum }}q_{j}D_j\left( x-t\right)\right\vert d\mu\left(t\right) \leq\frac{cM_k}{M_N}.
	\end{equation*}%
	Here $c$ is an absolute constant.
\end{lemma}

{\bf Proof}:
Let $x\in I_N^{k,l},$ for $0\leq k<l\leq N-1$ and $t\in I_N.$ First, we observe that $x-t\in $ $I_N^{k,l}.$ Next, we apply Lemmas \ref{gatkn} and  \ref{lemma0nnT} to obtain that
\begin{eqnarray}  \label{8881}
&&\int_{I_N}\left\vert\frac{1}{Q_n}\overset{n-1}{\underset{j=M_{N}}{\sum }}q_{j}D_j\left( x-t\right)\right\vert d\mu\left(t\right)  \\ \notag
&\leq & \frac{c}{n}\underset{i=0}{\overset{\left\vert n\right\vert }{\sum}}M_i\int_{I_N}\left\vert K_{M_i}\left(x-t\right)\right\vert d\mu
\left(t\right) \\ \notag
&\leq& \frac{c}{n}\int_{I_N}\overset{l}{\underset{i=0}{\sum }}M_iM_kd\mu\left(t\right) \leq \frac{cM_kM_l}{nM_N}
\end{eqnarray}
and the first estimate is proved.

Now, let $x\in I_N^{k,N}$. Since $x-t\in I_{N}^{k,N}$ for $t\in I_N,$ by
combining again Lemmas \ref{lemma2} and  \ref{lemma0nnT} we have that
\begin{eqnarray} \label{111100}
&&\int_{I_N}\left\vert\frac{1}{Q_n}\overset{n-1}{\underset{j=M_{N}}{\sum }}q_{j}D_j\left( x-t\right)\right\vert d\mu\left(t\right)   \\ \notag
&\leq&\frac{c}{n}\underset{i=0}{\overset{\left\vert n\right\vert } {\sum}}M_i\int_{I_N}\left\vert K_{M_i}\left(x-t\right)\right\vert d\mu\left(t\right) \\ \notag
&\leq&\frac{c}{n}\overset{\left\vert n\right\vert-1} {\underset{i=0}{\sum }}M_i\int_{I_N}M_kd\mu\left(t\right)\leq\frac{cM_k}{M_N}.
\end{eqnarray}

By combining (\ref{8881}) and (\ref{111100}) we complete the proof.

\begin{lemma} \label{lemma5bT}
	Let $n\geq M_N, \ \ x\in I_N^{k,l}, \ \ k=0,\dots,N-1, \ \
	l=k+1,\dots,N$ and $\{q_k:k\in\mathbb{N}\}$ be a sequence of
	non-increasing sequence, satisfying condition \eqref{fn0111}. Then
	\begin{equation*}
	\int_{I_N}\left\vert\frac{1}{Q_n}\overset{n-1}{\underset{j=M_{N}}{\sum }}q_{j}D_j\left( x-t\right)\right\vert d\mu\left(
	t\right)\leq\frac{cM_lM_k}{M_N^2},
	\end{equation*}%
	where $c$ is an absolute constant.
\end{lemma}

{\bf Proof}:
Since $n\geq M_N$ if we apply Lemma \ref{lemma5a} we immediately get the  proof.

\begin{lemma}\label{lemma0nnT1}
	Let $\{q_k:k\in\mathbb{N}\}$ be a sequence of non-decreasing numbers satisfying  (\ref{fn01}).
	Then
	\begin{equation*}
	\left\vert F_n\right\vert\leq\frac{c}{n}\left\{\sum_{j=0}^{\left\vert n\right\vert }M_j\left\vert K_{M_j}\right\vert \right\},
	\end{equation*}
	where $c$ is an absolute constant.
\end{lemma}

\begin{proof}
	Since sequence $\{q_k:k\in \mathbb{N}\}$ be non-decreasing if we apply condition \eqref{fn01} we can conclude that
	\begin{eqnarray*}
		&&\frac{1}{Q_n}\left(\overset{n-2}{\underset{j=0}{\sum}}\left\vert q_j-q_{j+1}\right\vert+q_{n-1}\right) \\ &&\leq\frac{1}{Q_n}\left(\overset{n-2}{\underset{j=0}{\sum }}\left (q_{j+1}-q_j \right)+q_{n-1}\right)\\
		&\leq& \frac{2q_{n-1}-q_0}{Q_{n}}\leq \frac{q_{n-1}}{Q_{n}}
		\leq \frac{c}{n}.
	\end{eqnarray*}
	
	If we apply \eqref{2c}  and \eqref{fn5} we immediately get that
	\begin{eqnarray*}
		\left\vert F_n\right\vert &\leq & \left( \frac{1}{Q_n}\left( \overset{n-1}{\underset{j=1}{\sum }}\left\vert q_{j}-q_{j+1} \right\vert+q_0\right)\right)\sum_{i=0}^{\left\vert n\right\vert } M_i\left\vert K_{M_i}\right\vert \\
		&=&\left(\frac{1}{Q_n}\left(\overset{n-1}{\underset{j=1}{\sum}}\left(q_{j}-q_{j+1}\right)+q_0\right)\right) \sum_{i=0}^{\left\vert n\right\vert}M_i\left\vert K_{M_i}\right\vert \\
		&\leq & \frac{q_{n-1}}{Q_n}\sum_{i=0}^{\left\vert n\right\vert}M_i\left\vert K_{M_i}\right\vert\leq\frac{c}{n}\sum_{i=0}^{\left\vert n\right\vert }M_i\left\vert K_{M_i}\right\vert.
	\end{eqnarray*}
	The proof is complete.
\end{proof}

\begin{lemma} \label{lemma5aT}
	Let $x\in I_N^{k,l}, \ \ k=0,\dots,N-2, \ \ l=k+1,\dots ,N-1$
	and $\{q_k:k\in \mathbb{N}\}$ be a sequence of non-decreasing numbers,
	satisfying condition (\ref{fn01}). Then
	\begin{equation*}
	\int_{I_N}\left\vert F_n\left(x-t\right)\right\vert d\mu\left(t\right) \leq\frac{cM_lM_k}{nM_N}.
	\end{equation*}
	Let $x\in I_N^{k,N},$ \ \ $k=0,\dots,N-1.$ Then
	\begin{equation*}
	\int_{I_N}\left\vert F_n\left(x-t\right)\right\vert d\mu\left(t\right) \leq\frac{cM_k}{M_N}.
	\end{equation*}%
	Here $c$ is an absolute constant.
\end{lemma}

{\bf Proof}:
The proof is quite analogously to Lemma \ref{lemma5a}. So we leave out the details.

\begin{lemma} \label{lemma5b}
	Let $n\geq M_N, \ \ x\in I_N^{k,l}, \ \ k=0,\dots,N-1, \ \
	l=k+1,\dots,N$ and $\{q_k:k\in\mathbb{N}\}$ be a sequence of
	non-decreasing sequence, satisfying condition (\ref{fn01}). Then
	\begin{equation*}
	\int_{I_N}\left\vert F_n\left(x-t\right)\right\vert d\mu\left(
	t\right)\leq\frac{cM_lM_k}{M_N^2}.
	\end{equation*}%
\end{lemma}

{\bf Proof}:
Since $n\geq M_N$ if we apply Lemma \ref{lemma5aT} we immediately get the  proof.

\section{Proofs of the Theorems}

\begin{proof} [Proof of Theorem \ref{theorem3fejermax2222}]
Let $0<p\leq 1/2$ and sequence $\{q_{k}:k\geq 0\}$ be non-increasing. By combining (\ref{2b}) and (\ref{2cc2})  we get that
\begin{eqnarray*}
	&&\widetilde{T}_{p}^{\ast }f:=\frac{\left\vert T_{n}f\right\vert}{{\left( n+1\right) ^{1/p-2}\log ^{2\left[ 1/2+p\right] }\left( n+1\right) }} \\
	&\leq &\frac{1}{{\left( n+1\right) ^{1/p-2}\log ^{2\left[ 1/2+p%
				\right] }\left( n+1\right) }}\left\vert \frac{1}{Q_{n}}\overset{n-1}{
		\underset{j=1}{\sum }}q_{j}S_{j}f\right\vert \\
	&\leq &\frac{1}{\left( n+1\right) ^{1/p-2}\log ^{2\left[ 1/2+p%
				\right] }\left( n+1\right) }\frac{1}{Q_{n}}\left( \overset{n-2}{\underset{j=1}{\sum }}\left\vert
	q_{j}-q_{j+1}\right\vert j\left\vert \sigma _{j}f\right\vert
	+q_{n-1}(n-1)\left\vert \sigma _{n}f\right\vert \right) \\
	&\leq &\frac{1}{Q_{n}}\left( \overset{n-2}{\underset{j=1}{\sum }} \frac{\left\vert
		q_{j}-q_{j+1}\right\vert j\left\vert \sigma _{j}f\right\vert}{{\left( j+1\right)^{1/p-2}\log ^{2\left[ 1/2+p\right] }\left( j+1\right) }}
	+\frac{q_{n-1}(n-1)\left\vert \sigma _{n}f\right\vert}{\left( n+1\right)^{1/p-2}\log ^{2\left[ 1/2+p\right] }\left( n+1\right) } \right) \\
	&\leq &\frac{1}{Q_{n}}\left( \overset{n-2}{\underset{j=1}{\sum }}\left(
	q_{j}-q_{j+1}\right) j+q_{n-1}(n-1)\right) \sup_{n\in \mathbb{N}_{+}}\frac{\left\vert
		\sigma _{n}f\right\vert }{\left( n+1\right) ^{1/p-2}\log ^{2\left[ 1/2+p%
			\right] }\left( n+1\right) } \\
	&\leq& \sup_{n\in \mathbb{N}_{+}}
		\frac{\left\vert
		\sigma _{n}f\right\vert }{\left( n+1\right) ^{1/p-2}\log ^{2\left[ 1/2+p
		\right] }\left( n+1\right) }:=\widetilde{\sigma }_{p}^{\ast }f,
\end{eqnarray*}
so that
$
\widetilde{T}_{p}^{\ast }f\leq \widetilde{\sigma }_{p}^{\ast }f.  \label{12aaaa}
$ Hence, if we apply
Theorem T1 we can conclude that the maximal operators  $\widetilde{T}_{p}^{\ast }$ of  $T$ means with non-increasing sequence $\{q_{k}:k\geq 0\}$ are bounded from the Hardy space $H_{p}$ to the space $L_{p}$ for $0<p\leq 1/2$. The proof  is complete.
\end{proof}

\begin{proof} [Proof of Theorem \ref{theorem3fejermax22221}]
Let $0<p\leq 1/2$ and sequence $\{q_{k}:k\geq 0\}$ be non-decreasing satisfying the condition \eqref{fn01}. By combining (\ref{2b}) and (\ref{2cc2})  we find that
	\begin{eqnarray*}
		&&\widetilde{T}_{p}^{\ast }f:=\frac{\left\vert T_{n}f\right\vert}{{\left( n+1\right) ^{1/p-2}\log ^{2\left[ 1/2+p\right] }\left( n+1\right) }} \\
		&\leq &\frac{1}{{\left( n+1\right) ^{1/p-2}\log ^{2\left[ 1/2+p%
					\right] }\left( n+1\right) }}\left\vert \frac{1}{Q_{n}}\overset{n-1}{
			\underset{j=1}{\sum }}q_{j}S_{j}f\right\vert \\
		&\leq &\frac{1}{\left( n+1\right) ^{1/p-2}\log ^{2\left[ 1/2+p%
				\right] }\left( n+1\right) }\frac{1}{Q_{n}}\left( \overset{n-2}{\underset{j=1}{\sum }}\left\vert
		q_{j}-q_{j+1}\right\vert j\left\vert \sigma _{j}f\right\vert
		+q_{n-1}(n-1)\left\vert \sigma _{n}f\right\vert \right) \\
		&\leq &\frac{1}{Q_{n}}\left( \overset{n-2}{\underset{j=1}{\sum }} \frac{\left\vert
			q_{j}-q_{j+1}\right\vert j\left\vert \sigma _{j}f\right\vert}{{\left( j+1\right)^{1/p-2}\log ^{2\left[ 1/2+p\right] }\left( j+1\right) }}
		+\frac{q_{n-1}(n-1)\left\vert \sigma _{n}f\right\vert}{\left( n+1\right)^{1/p-2}\log ^{2\left[ 1/2+p\right] }\left( n+1\right) } \right) \\
		&\leq &\frac{1}{Q_{n}}\left( \overset{n-2}{\underset{j=1}{\sum }}\left(
		q_{j+1}-q_{j}\right) j+q_{n-1}(n-1)\right) \sup_{n\in \mathbb{N}_{+}}\frac{\left\vert
			\sigma _{n}f\right\vert }{\left( n+1\right) ^{1/p-2}\log ^{2\left[ 1/2+p%
				\right] }\left( n+1\right) } \\
		&\leq& \frac{2q_{n-1}(n-1)-Q_n}{Q_{n}}\sup_{n\in \mathbb{N}_{+}}
		\frac{\left\vert
			\sigma _{n}f\right\vert }{\left( n+1\right) ^{1/p-2}\log ^{2\left[ 1/2+p
				\right] }\left( n+1\right) }\\
		&\leq& \sup_{n\in \mathbb{N}_{+}}
		\frac{\left\vert
			\sigma _{n}f\right\vert }{\left( n+1\right) ^{1/p-2}\log ^{2\left[ 1/2+p
				\right] }\left( n+1\right) }=\widetilde{\sigma }_{p}^{\ast }f.
	\end{eqnarray*}

	so that
	\begin{equation}
	\widetilde{T}_{p}^{\ast }f\leq \widetilde{\sigma }_{p}^{\ast }f  \label{12aaaaa}
	\end{equation}
	
	If we apply (\ref{12aaaaa}) and Theorem T1 we can conclude that the maximal operators  $\widetilde{T}_{p}^{\ast }$ of  $T$ means with non-decreasing sequence $\{q_{k}:k\geq 0\},$ are bounded from the Hardy space $H_{p}$ to the space $L_{p}$ for $0<p\leq 1/2.$ The proof   is complete.
\end{proof}

\begin{proof}[Proof of Theorem \protect\ref{theorem2fejerstrong}] Let $0<p< 1/2$ and sequence $\{q_{k}:k\geq 0\}$ be non-increasing. By Lemma \ref{lemma2.1}, the proof of part a)  will be complete, if we show that
	\begin{equation*}
	\overset{\infty}{\underset{m=1}{\sum }}%
	\frac{\left\Vert T_{m}a\right\Vert _{H_{p}}^{p}}{m^{2-2p}}\leq c_{p},
	\end{equation*}
	for every $p$-atom $a,$ with support$\ I$, $\mu \left( I\right) =M_{N}^{-1}.$
	We may assume that $I=I_{N}.$ It is easy to see that $S_{n}\left( a\right)
	=T_{n}\left( a\right) =0,$ when $n\leq M_{N}$. Therefore, we can suppose
	that $n>M_{N}$.
	
	Let $x\in I_{N}.$ Since $T_{n}$ is bounded from $L_{\infty }$ to $L_{\infty
	} $ (boundedness follows Lemma \ref{T2}) and $\left\Vert a\right\Vert
	_{\infty }\leq M_{N}^{1/p}$ we obtain that 
	\begin{equation*}
	\int_{I_{N}}\left\vert T_{m}a\right\vert ^{p}d\mu \leq \frac{\left\Vert
		a\right\Vert _{\infty }^{p}}{M_{N}}\leq c<\infty. 
	\end{equation*}%
	Hence, 
	\begin{equation} \label{14b14}
	\overset{\infty}{\underset{m=1}{\sum }}%
	\frac{\int_{I_{N}}\left\vert T_{m}a\right\vert ^{p}d\mu }{m^{2-2p}}\leq 
	\overset{\infty}{\underset{k=1}{\sum }}%
	\frac{1}{m^{2-2p}}\leq c<\infty ,\text{ \ }0<p< 1/2.  
	\end{equation}
	
	It is easy to see that 
	\begin{eqnarray}\label{14a14}
	\left\vert T_{m}a\left( x\right) \right\vert &=&\left\vert\int_{I_{N}} a\left(t\right) F_{n}\left( x-t\right)  d\mu \left( t\right)\right\vert\\ \notag
	&=&\left\vert\int_{I_{N}} a\left( t\right) \frac{1}{Q_{n}}\overset{n}{\underset%
		{j=M_{N}}{\sum }}q_{j}D_{j}\left( x-t\right)  d\mu \left(
	t\right) \right\vert  \\ \notag
	&\leq& \left\Vert a\right\Vert _{\infty }\int_{I_{N}}\left\vert \frac{1}{Q_{n}}%
	\overset{n}{\underset{j=M_{N}}{\sum }}q_{j}D_{j}\left( x-t\right)
	\right\vert d\mu \left( t\right) \\ \notag
	&\leq& M_{N}^{1/p}\int_{I_{N}}\left\vert 
	\frac{1}{Q_{n}}\overset{n}{\underset{j=M_{N}}{\sum }}q_{j}D_{j}\left(
	x-t\right) \right\vert d\mu \left( t\right) 
	\end{eqnarray}
	
	Let $T_{n}$ be $T$ means, with non-decreasing coefficients $%
	\{q_{k}:k\geq 0\}$ and $x\in I_{N}^{k,l},\,0\leq k<l\leq N.$ Then, in the
	view of Lemma \ref{lemma5aa} we get that 
	\begin{equation} \label{12q1}
	\left\vert T_{m}a\left( x\right) \right\vert \leq cM_{l}M_{k}M_{N}^{1/p-2},%
	\text{ for }0<p< 1/2.  
	\end{equation}
	
	Let $0<p<1/2.$ By using (\ref{1.1}), (\ref{14a14})  and (\ref{12q1}) we find that%
	\begin{eqnarray}\label{7aaa}
	\int_{\overline{I_{N}}}\left\vert T_{m}a\right\vert ^{p}d\mu &=&\overset{N-2}{%
		\underset{k=0}{\sum }}\overset{N-1}{\underset{l=k+1}{\sum }}%
	\sum\limits_{x_{j}=0,\text{ }j\in \{l+1,\dots
		,N-1\}}^{m_{j-1}}\int_{I_{N}^{k,l}}\left\vert T_{m}a\right\vert ^{p}d\mu +%
	\overset{N-1}{\underset{k=0}{\sum }}\int_{I_{N}^{k,N}}\left\vert
	T_{m}a\right\vert ^{p}d\mu   \\ \notag
	&\leq& c\overset{N-2}{\underset{k=0}{\sum }}\overset{N-1}{\underset{l=k+1}{%
			\sum }}\frac{m_{l+1}\dotsm m_{N-1}}{M_{N}}\left( M_{l}M_{k}\right)
	^{p}M_{N}^{1-2p}+\overset{N-1}{\underset{k=0}{\sum }}\frac{1}{M_{N}}%
	M_{k}^{p}M_{N}^{1-p}  \\ \notag
	&\leq & cM_{N}^{1-2p}\overset{N-2}{\underset{k=0}{\sum }}\overset{N-1}{\underset%
		{l=k+1}{\sum }}\frac{\left( M_{l}M_{k}\right) ^{p}}{M_{l}}+\overset{N-1}{%
		\underset{k=0}{\sum }}\frac{M_{k}^{p}}{M_{N}^{p}}\leq cM_{N}^{1-2p}.
	\end{eqnarray}
	
	Moreover, according to (\ref{7aaa}) we get that
	\begin{equation} \label{14b15}
	\overset{\infty }{\underset{m=M_{N}+1}{\sum }}\frac{\int_{\overline{I_{N}}%
		}\left\vert T_{m}a\right\vert ^{p}d\mu }{m^{2-2p}}\leq \overset{\infty }{%
		\underset{m=M_{N}+1}{\sum }}\frac{cM_{N}^{1-2p}}{m^{2-2p}}<c<\infty ,\text{
		\ }\left( 0<p<1/2\right) .
	\end{equation}%
	The proof of part a) is complete by just combining \eqref{14b14} and \eqref{14b15}.

 Let $p=1/2$ and $T_{n}$ be $T$ means, with non-increasing coefficients $%
	\{q_{k}:k\geq 0\}$, satisfying condition (\ref{fn0111}). By Lemma \ref{lemma2.1}, the proof of part b)  will
	be complete, if we show that%
	\begin{equation*}
	\frac{1}{\log n}\overset{n}{\underset{m=1}{\sum }}%
	\frac{\left\Vert T_{m}a\right\Vert _{H_{1/2}}^{1/2}}{m}\leq c,
	\end{equation*}%
	for every $1/2$-atom $a,$ with support$\ I$, $\mu \left( I\right) =M_{N}^{-1}.$
	We may assume that $I=I_{N}.$ It is easy to see that $S_{n}\left( a\right)
	=T_{n}\left( a\right) =0,$ when $n\leq M_{N}$. Therefore, we can suppose
	that $n>M_{N}$.
	
	Let $x\in I_{N}.$ Since $T_{n}$ is bounded from $L_{\infty }$ to $L_{\infty
	} $ (boundedness follows from Lemma \ref{T2}) and $\left\Vert a\right\Vert
	_{\infty }\leq M_{N}^{2}$ we obtain that 
	\begin{equation*}
	\int_{I_{N}}\left\vert T_{m}a\right\vert ^{1/2}d\mu \leq \frac{\left\Vert
		a\right\Vert _{\infty }^{1/2}}{M_{N}}\leq c<\infty.
	\end{equation*}%
	Hence, 
	\begin{equation}
	\frac{1}{\log n}\overset{n}{\underset{m=1}{\sum }}%
	\frac{\int_{I_{N}}\left\vert T_{m}a\right\vert ^{1/2}d\mu }{m}\leq 
	\frac{1}{\log n}\overset{n}{\underset{k=1}{\sum }}%
	\frac{1}{m}\leq c<\infty .  \label{14b}
	\end{equation}
	
Analogously to \eqref{14a14} we find that
\begin{eqnarray} \label{14aa} 
\left\vert T_{m}a\left( x\right) \right\vert &=&\left\vert\int_{I_{N}} a\left( t\right) \frac{1}{Q_{n}}\overset{n}{\underset%
	{j=M_{N}}{\sum }}q_{j}D_{j}\left( x-t\right)  d\mu \left(
t\right) \right\vert  \\ \notag
&\leq& \left\Vert a\right\Vert _{\infty }\int_{I_{N}}\left\vert F_{m}\left(
x-t\right) \right\vert d\mu \left( t\right) 
\leq
M_{N}^{2}\int_{I_{N}}\left\vert F_{m}\left( x-t\right) \right\vert d\mu
\left( t\right) .
\end{eqnarray}

Let $x\in I_{N}^{k,l},\,0\leq k<l<N.$ Then, in the view of Lemma \ref{lemma5a}
we get that 
\begin{equation}
\left\vert T_{m}a\left( x\right) \right\vert \leq \frac{cM_{l}M_{k}M_{N}}{m}.
\label{13q1}
\end{equation}

Let $x\in I_{N}^{k,N}.$ Then, according to Lemma \ref{lemma5a} we obtain that 
\begin{equation}
\left\vert T_{m}a\left( x\right) \right\vert \leq cM_{k}M_{N}.  \label{13q2}
\end{equation}

By combining (\ref{1.1}), (\ref{14aa}), (\ref{13q1}) and (\ref{13q2}) we
obtain that%
\begin{eqnarray*}
&&\int_{\overline{I_{N}}}\left\vert T_{m}a\left( x\right) \right\vert
^{1/2}d\mu \left( x\right)
\\
&\leq& c\overset{N-2}{\underset{k=0}{\sum }}\overset{N-1}{\underset{l=k+1}{%
\sum }}\frac{m_{l+1}\dotsm m_{N-1}}{M_{N}}\frac{\left( M_{l}M_{k}\right)
^{1/2}M_{N}^{1/2}}{m^{1/2}}+\overset{N-1}{\underset{k=0}{\sum }}\frac{1}{%
M_{N}}M_{k}^{1/2}M_{N}^{1/2}  \\ 
&\leq& M_{N}^{1/2}\overset{N-2}{\underset{k=0}{\sum }}\overset{N-1}{\underset{
l=k+1}{\sum }}\frac{\left( M_{l}M_{k}\right) ^{1/2}}{m^{1/2}M_{l}}+\overset{
N-1}{\underset{k=0}{\sum }}\frac{M_{k}^{1/2}}{M_{N}^{1/2}}\leq \frac{%
cM_{N}^{1/2}N}{m^{1/2}}+c.
\end{eqnarray*}

It follows that%
\begin{eqnarray}\label{15b}
&&\frac{1}{\log n}\overset{n}{\underset{m=M_{N}+1}{\sum }}\frac{\int_{%
\overline{I_{N}}}\left\vert T_{m}a\left( x\right) \right\vert ^{1/2}d\mu
\left( x\right) }{m}\leq \frac{1}{\log n}\overset{n}{\underset{m=M_{N}+1}{%
\sum }}\left( \frac{cM_{N}^{1/2}N}{m^{3/2}}+\frac{c}{m}\right) <c<\infty .
\end{eqnarray}

The proof of part b) is completed by just combining (\ref{14b}) and (\ref%
{15b}).
\end{proof}

\begin{proof}[Proof of Theorem \protect\ref{theorem2fejerstrong1}]
If we use Lemmas \ref{lemma5aT} and \ref{lemma5b} and follows analogical steps of proof of   Theorem \ref{theorem2fejerstrong} we immediately get the proof of Theorem \protect\ref{theorem2fejerstrong1}. So, we leave out the details.
\end{proof}

\end{document}